\documentclass[11pt,reqno]{amsart}

 \usepackage{amsmath,amsthm,amsfonts}
\usepackage{latexsym,mathabx,txfonts}
\usepackage{amssymb}
\usepackage{cite}

\usepackage{color,soul} 

\newcommand*\changed{}

\newcommand{\R}{\mathbb{R}}

\numberwithin{equation}{section}

\newtheorem{thm}{Theorem}[section]

\newtheorem{lem}[thm]{Lemma}

\theoremstyle{remark}
\newtheorem{rem}{Remark}[section]

\newcommand{\Del}[1]{}

\newcommand{\supp}{\operatorname{supp}}

\def\eps{\varepsilon}

\begin{document}

\title[Type I blow up]{On type I blow up formation for the critical NLW}

\author{Joachim Krieger, Willie Wong}

\subjclass[2010]{35L05, 35B44}

\keywords{critical wave equation, hyperbolic dynamics,  blowup, scattering, stability, invariant manifold}

\thanks{Support of the Swiss National Fund for
the first author is gratefully acknowledged.}

\begin{abstract}
\changed{We introduce a suitable concept of weak evolution in the context of
the radial quintic focussing semilinear wave equation on $\R^{3+1}$,
that is adapted to continuation past type II singularities. We show
that the weak extension leads to type I singularity formation for
initial data corresponding to: (i) the Kenig-Merle blow-up solutions
with initial energy below the ground state and (ii) the
Krieger-Nakanishi-Schlag blow-up solutions sitting initially near and
``above'' the ground state static solution.}
\end{abstract}

\maketitle

\section{Introduction}
We consider the critical focussing nonlinear wave equation on $\R^{3+1}$, given by 
\begin{equation}\label{eq:Main}
\Box u := -u_{tt} + \triangle u = -u^5\,,
\end{equation}
which has a (possibly negative) conserved energy 
\[
E(u): = \int_{\R^3}\big(\frac{1}{2}(u_t^2 +|\nabla_x u|^2) -
\frac{u^6}{6}\big)\,\mathrm{d}x\,.
\]
We restrict to radial solutions of the form $u(t, x) = v(t, |x|)$. It
is well-known and easy to show that this model admits finite time blow
up solutions \changed{with finite initial free energy}
\[ 
E_{\text{free}}(u)(t): = \int_{\R^3}\big[\frac{1}{2}(u_t^2 +
|\nabla_x u|^2)\big]\,\mathrm{d}x\,.
\]
\changed{One can start with the explicit ODE-type solutions}
\[
u(t, x) = \frac{(\frac{3}{4})^{\frac{1}{4}}}{(T-t)^{\frac{1}{2}}}
\]
\changed{for any $T\in \R_+$.} By truncation to a backward (or
forward) light cone and invocation of Huygens' principle, one can modify 
these to solutions \changed{for which the initial free energy is
finite}. Indeed, one may
consider data $u[0] = \big(u(0, \cdot), u_t(0, \cdot)\big)$ with 
\[
u(0, \cdot) =
\chi_{|x|<3T}\frac{(\frac{3}{4})^{\frac{1}{4}}}{T^{\frac{1}{2}}}\,,\quad u_t(0, \cdot) =  \chi_{|x|<3T}\frac{(\frac{3}{64})^{\frac{1}{4}}}{T^{\frac{3}{2}}}
\]
where the cut-off function $\chi_{|x|<3T}|_{|x|\leq 2T} = 1$ and smoothly truncates to the region $|x|<3T$. Observe that these solutions satisfy 
\[
\limsup_{t\nearrow T} \int_{|x| < T} \big[\frac12(u_t^2(t) + |\nabla_x
u(t)|^2)\big]\,\mathrm{d}x = +\infty\,,
\]
\changed{and thus cannot be continued past time $T$ (though a
singularity may form at some earlier time, depending on the
choice of cut-off $\chi_{|x|<3T}$).}

\changed{Motivated by these ODE type blow-ups, we say a blow-up
solution $u$ with maximum forward time of existence $T$ is of 
\emph{type I} if}
\[ \limsup_{t\nearrow T} E_{\text{free}}(u)(t) = +\infty \]
\changed{and \emph{type II} otherwise, that is}
\[
\limsup_{t\nearrow T}E_{\text{free}}(u)(t) < +\infty\,.
\]
Recent works by Duyckaerts-Kenig-Merle \cite{DKM1, DKM2, DKM3, DKM4}
have provided a complete classification near the blow-up time of type 
II solutions for
\eqref{eq:Main}, while existence of solutions of this type was
established in \cite{KST} and \cite{DoKr}. 
Here, we would like to discuss the formation of type I blow up. To the
best of the authors' knowledge, previously demonstrated type I blow up 
mechanisms all derive in principle from $u_t$ having a sign pointwise. 
In addition to the explicit ODE solutions (and perturbations thereof
as in \cite{DoSch}), Duyckaerts-Kenig-Merle showed in
\cite{DKM3} that monotonicity in time of a radial solution
close to the blow up time implies type I blow up, which they then used to
show that the $W^+$ solution of \cite{DM2} as well as solutions given
by initial data $u[0] = (cW, 0)$ with $c > 1$ all evolve into type I
blow ups. 

In a recent work \cite{CNLW3}, the study of all possible dynamics
which result as {\it{perturbations of the static solution}} $W(x) =
\left(1+\frac{|x|^2}{3}\right)^{-\frac{1}{2}}$ was begun. Note that these
static solutions are a special feature of the energy critical case.
Also, crucially for the analysis of \cite{CNLW3}, the perturbations
are close to $W$ with respect to a norm strictly stronger than the
energy. It was then shown in \cite{CNLW3} that there exists a
co-dimension one Lipschitz manifold $\Sigma$ passing through
$W$ such that within a sufficiently close neighbourhood to $W$, data
`above' $\Sigma$ result in finite time blow up while data `below'
$\Sigma$ scatter to zero, all in forward time. Further, data precisely
located on $\Sigma$ lead to solutions in forward time scattering toward a re-scaling of $W$. 

In this note, we would like to study the finite-time blow up solutions
corresponding to data slightly above $\Sigma$. Conjecturally, a
generic set within these solutions ought to correspond to type I blow
up solutions. At this time we cannot show this. Instead, our goal here
is to introduce a suitable concept of \emph{canonical weak solution}
and show that such solutions will result eventually, in finite time,
in a type I blow up scenario. This will be seen to directly result
from a combination of the recent breakthrough characterization of type
II blow up solutions by Duyckaerts-Kenig-Merle \cite{DKM4} with the
techniques developed in \cite{CNLW1}.  \changed{Along the way, we will
also show that the canonical weak extensions of the blow-up solutions 
exhibited by Kenig-Merle} \cite{KM2}, \changed{whose initial energy 
is below that of the ground state, terminates in finite time with
exploding free energy.}

\changed{The authors would like to thank the anonymous referee for the detailed
comments which improved the manuscript.}

\section{A canonical concept of weak evolution}\label{sec:canon}

Let $u(t, x)$ be a (radial) Shatah-Struwe energy class solution (see
e.g. \cite{ShaStr}, \cite{KM2}) of \eqref{eq:Main}, existing on an
interval $I =[0, T)$, $T>0$. Also, assume that $I$ is a maximal such
interval. If $T<\infty$, then the solution either has a type I
singularity at $T$, or else a type II singularity. Assume the latter
situation. According to the seminal work \cite{DKM4}, the solution
admits a decomposition (writing $W_\lambda(x) = \lambda^\frac12
W(\lambda x)$ for the $\dot{H}^1$ invariant scaling)
\begin{gather*}
u(t, \cdot) = \sum_{i=1}^{N}\kappa_iW_{\lambda_i(t)}(\cdot) + u_1(t,
\cdot) + o_{\dot{H}^1}(1),\quad \kappa_i\in \{\pm 1\},\\
u_t(t, \cdot) = u_{1,t}(t, \cdot) + o_{L^2}(1)
\end{gather*}
as $t\nearrow T$, with\footnote{The notation $a(t)\ll b(t)$ here
means $\lim_{t\nearrow T}\frac{b(t)}{a(t)} = +\infty$.} 
\[
(T-t)^{-1}\ll \lambda_1(t)\ll \lambda_2(t)\ll\ldots\ll \lambda_N(t)
\]
and $u_1(t, \cdot)$ is an energy class solution of \eqref{eq:Main}
in a neighborhood around time $t = T$. One easily verifies that this $u_1(t, \cdot)$ is indeed uniquely determined by $u$ and $T$. 
It is then natural, assuming that there is a type II singularity at time $t = T$, to continue the evolution past time $T$ by imposing data 
\[
u[T] = \big(u(T, \cdot), u_t(T, \cdot)\big): = \big(u_1(T, \cdot),
u_{1,t}(T, \cdot)\big)
\]
and then using the Shatah-Struwe evolution of $u[T]$ starting from
time $T =:T_1$. Then there exists $T_2\in (T_1, +\infty]$, such that
if $T_2<\infty$, there is either a type I or type II singularity at
$T_2$, and then in the latter case again the Duyckaerts-Kenig-Merle profile decomposition applies at time $t = T_2$, allowing us to write 
\[
u(t, \cdot) =
\sum_{i=1}^{N_2}\kappa_i^{(2)}W_{\lambda_i^{(2)}(t)}(\cdot) + u_2(t,
\cdot) + o_{\dot{H}^1}(1),\,\,t\nearrow T_2,
\]
where $u_2$ is now a solution of \eqref{eq:Main} in a neighborhood
containing $t = T_2$. In this way, we obtain a sequence of times 
\[
T_1<T_2<T_3<\ldots,\,T_i\in (0, +\infty]
\]
with the following possibilities: (i) the sequence is finite, and the
last $T_{\text{terminal}}:= T_N = +\infty$, with all previous $T_i$ being 
type II blow up times; (ii) the sequence is finite, and the last
$T_{\text{terminal}}:= T_N < \infty$ being the first type I blow up
time in the evolution; (iii) the sequence is infinite and we define
$T_{\text{terminal}}:= \lim_{i\to\infty} T_i \in (0,\infty]$. Note
that that except in case (ii) we have no \emph{a priori} knowledge as
to whether the solution blows up at $T_{\text{terminal}}$.

We now define the \emph{canonical evolution} of the data $u[0]$ on
$[0,T_{\text{terminal}})$ to be the function $\tilde{u}(t,\cdot)$ 
given by $u(t, \cdot)$ on $[0, T_1)$, by $u_1(t, \cdot)$ on $[T_1,
T_2)$ etc.  

On the other hand, we define $u(t, x)\in L^\infty([0,T_*),
\dot{H}^1)\cap W^{1, \infty}([0,T_*), L^2)$ to be a \emph{weak
solution} of \eqref{eq:Main}, provided for every $\phi\in
C_0^\infty((-\infty, T_*)\times\R^{3})$, we have 
\begin{multline}\label{eq:weak}
\int_{\R^3}u_t(0, \cdot)\phi(0, \cdot)\,\mathrm{d}x +
\int_0^{T_*}\int_{\R^3}\big(u_t\phi_t - \nabla_x
u\cdot\nabla_x\phi\big)\,\mathrm{d}x\,\mathrm{d}t\\
 = -\int_0^{T_*}\int_{\R^3}u^5\phi\,\mathrm{d}x\,\mathrm{d}t\,.
\end{multline}
Note that our concept of canonical weak evolution is in fact more regular than $ L^\infty([0,T_*), \dot{H}^1)\cap W^{1, \infty}([0,T_*), L^2)$, since  
\[
\tilde{u}|_{[T_{i-1}, T_i)}\in C^0\big([T_{i-1}, T_i),
\dot{H}^1\big)\cap C^1\big([T_{i-1}, T_i), L^2\big).
\]
In particular, for the canonical evolution $\tilde{u}$ is right-continuous
at time 0, that is
$\lim_{t\searrow 0} \tilde{u}(t,\cdot) = u(0, \cdot)$ with respect to $\dot{H}^1$. 
Then 
\begin{lem}Let $\tilde{u}(t, \cdot)$ be the canonical evolution of
$u[0]\in \dot{H}^1\times L^2$, defined on $[0, T_{\text{terminal}})$.
Then $\tilde{u}$ is a weak solution of \eqref{eq:Main} in the above
sense with $T_* = T_{\text{terminal}}$. 
\end{lem}
\begin{proof} Let $\phi\in C_0^\infty((-\infty,
T_{\text{terminal}})\times\R^{3})$. Then recalling the construction of
$\tilde{u}$, there exist finitely many $T_i$, $i = 1, 2,\ldots, k$,
$T_0: = 0$, with $T_i\in \pi_t(\supp(\phi))$ with $\pi_t:
\R^{3+1}\rightarrow \R$ the projection onto the time coordinate. Then
we have $\tilde{u}|_{[T_i, T_{i+1})} = u_i$, $u_0 = u$ being the
evolution of the data $u[0]$. Now for each $i$, pick a function
$\chi\in C_0^\infty([T_i, T_{i+1}))$ with $\chi(T_i) = 1$; then
integrating by parts the Shatah-Struwe energy class solution $u_i$ we
have
\begin{multline}\label{eq:weaki}
\int_{\R^3}u_{i,t}(T_i, \cdot)\phi(T_i, \cdot)\,\mathrm{d}x +
\int_{T_i}^{T_{i+1}}\int_{\R^3}\big(u_{i,t}(\chi\phi)_{t} - \nabla_x
u_i\cdot\nabla_x(\chi\phi)\big)\,\mathrm{d}x\,\mathrm{d}t\\
 =
-\int_{T_i}^{T_{i+1}}\int_{\R^3}u_i^5\chi\phi\,\mathrm{d}x\,\mathrm{d}t
\end{multline}
Pick a sequence $\chi^{(k)}\in C_0^\infty([T_i, T_{i+1}))$ with
$\chi^{(k)}\rightarrow \chi_{[T_i, T_{i+1})}$ pointwise and locally uniformly and such that 
\[
\lim_{k\rightarrow\infty}\int_{T_i}^{T_{i+1}}(\chi^{(k)})'(t)f(t)\,dt = -f(T_{i+1})
\]
for $f\in C^0([T_i, T_{i+1}])$. Since we can write (as $t\nearrow
T_{i+1}$ and where $\kappa_k^{(i)} \in \{\pm 1\}$)
\[
u_i(t, \cdot) \xrightarrow{\dot{H}^1} \sum_{k=1}^{N_i}\kappa_k^{(i)}
W_{\lambda_k^{(i)}(t)}(\cdot) + u_{i+1}(t,
\cdot),\qquad u_{i,t}\xrightarrow{L^2} u_{i+1,t}(T_{i+1}, \cdot),
\]
we infer 
\begin{multline*}
\lim_{k\rightarrow\infty}\int_{T_i}^{T_{i+1}}\int_{\R^3}\big(u_{i,t}(\chi^{(k)}\phi)_{t}
- \nabla_x
  u_i\cdot\nabla_x(\chi^{(k)}\phi)\big)\,\mathrm{d}x\,\mathrm{d}t\\ 
= \int_{T_i}^{T_{i+1}}\int_{\R^3}\big(u_{i,t}\phi_{t} - \nabla_x
u_i\cdot\nabla_x(\phi)\big)\,\mathrm{d}x\,\mathrm{d}t -
\int_{\R^3}u_{i+1, t}(T_{i+1},\cdot)\phi(T_{i+1},\cdot)\,\mathrm{d}x
\end{multline*}
and so we obtain 
\begin{multline}\label{eq:localweak}
 \int_{T_i}^{T_{i+1}}\int_{\R^3}\big(u_{i,t}\phi_{t} - \nabla_x
u_i\cdot\nabla_x(\phi)\big)\,\mathrm{d}x\,\mathrm{d}t  +
\int_{T_i}^{T_{i+1}}\int_{\R^3}u_i^5\phi\,\mathrm{d}x\,\mathrm{d}t\\
= \int_{\R^3}u_{i+1, t}(T_{i+1},\cdot)\phi(T_{i+1},\cdot)\,\mathrm{d}x -
\int_{\R^3}u_{i,t}(T_i, \cdot)\phi(T_i, \cdot)\,\mathrm{d}x \,.
\end{multline}
Summation of the relations \eqref{eq:localweak} over $i = 1, 2,\ldots, k$, we find the relation 
\begin{equation}\begin{split}
&\int_{\R^3}\tilde{u}_{t}(0, \cdot)\phi(0, \cdot)\,dx +
\int_{0}^{T_{\text{terminal}}}\int_{\R^3}\big(\tilde{u}_{t}\phi_{t} - \nabla_x\tilde{u}\cdot\nabla_x\phi\big)\,dxdt\\
& = -\int_{0}^{T_{\text{terminal}}}\int_{\R^3}\tilde{u}^5\phi\,dxdt,
\end{split}\end{equation}
proving the lemma.
\end{proof}

\begin{rem}
\changed{Our definition of the canonical weak solution bears some
similarity to the \emph{semi-Strichartz} solutions defined by Tao
for the focussing nonlinear Sch\"odinger equation} \cite{TaoSS}.
\changed{Tao used the notion of semi-Strichartz solutions to bridge
the gap between availability of global existence results of weak
solutions and uniqueness results of Strichartz-class solutions. Here
we use the notion of canonical weak solutions to bridge the gap
between free energy explosion (type I blow-up) and finite time
singularity formation.}
\end{rem}

\changed{An important consequence of the profile decomposition of
Duyckaerts-Kenig-Merle} \cite{DKM4} \changed{is that, per the asymptotic
separation of profiles and energy conservation of the regular
evolution, the energy of our canonical weak evolution is \emph{strictly
decreasing}. In fact, we have that}
\begin{equation}\label{eq:energyjump}
E(u_i) = N_i\cdot E(W) + E(u_{i+1})~, \qquad N_i \geq 1 
\end{equation}
\changed{since the soliton energy is scale invariant. Evaluating}
\[ E(W) = \frac{1}{3}\int_{\R^3}|\nabla W|^2\,\mathrm{d}x > 0
\]
\changed{we see explicitly the energy jump as solitons get bubbled
off. This also implies that our canonical weak evolution concept is
time \emph{irreversible}.}

\section{Formation of type I singularities: general case}
Before considering the blow-up solutions of \cite{CNLW3},
we first prove some general lemmas about the eventual
formation of type I singularities for our canonical weak evolution.

\begin{lem}\label{lem:fti}
If a canonical weak solution $\tilde{u}$ satisfies
$T_{\text{terminal}} < +\infty$, then it satisfies the ``type I'' condition
\begin{equation}\label{eq:weaktypeiblowup}
\limsup_{t\nearrow T_{\text{terminal}}}E_{\text{free}}(\tilde{u})(t) =
+\infty~.
\end{equation}
\end{lem}
\begin{proof}
As discussed in Section \ref{sec:canon}, the only possibility when
$T_{\text{terminal}} < +\infty$ is either $T_{\text{terminal}} = T_N <
\infty$ ending in a type I blow-up, or there exists an infinite
sequence $T_i \nearrow T_{\text{terminal}}$ of type II blow-up points.
In the first case \eqref{eq:weaktypeiblowup} follows by definition. In
the second case we appeal to the energy evolution
\eqref{eq:energyjump} which implies that $\lim_{i\to\infty} E(u_i) =
-\infty$. Then from Sobolev's embedding we obtain, as claimed,
\[
\lim_{i\rightarrow\infty}\|\nabla_{t,x}u_i(T_i, \cdot)\|_{L_x^2} =
+\infty~.
\]
\end{proof}

Next, we show one of the main advantages of our canonical weak
solution construction: it preserves the virial type functional used in
\cite{CNLW1}. Define, for some large large $\tau>0$ to be fixed later, and 
a cutoff $\chi\in C_0^{\infty}([0,\infty))$ with $\chi|_{[0,1]} = 1$,
the functions 
\begin{equation}\label{eq:defy}
w(t, x) = \chi\Big(\frac{|x|}{t+\tau}\Big)\,,\qquad y(t): = \langle
w(t,\cdot)\tilde{u}(t,\cdot), \tilde{u}(t,\cdot)\rangle\,
\end{equation}
Here the spatial $L^2$ pairing $\langle\cdot,\cdot\rangle$ is well-defined as
long as $\tilde{u}\in\dot{H}^1$, due to Sobolev's embedding and the
cutoff. This is definitely the case on each interval $I_j: = (T_{i-1}, T_i)$. 
In fact, from the computations in  \cite[Section 5]{CNLW1}, we have that 
$y(t), \dot{y}(t), \ddot{y}(t)$ are continuous functions on each open interval $I_j$. 
Now suppose that $\tilde{u}$ is a canonical weak solution maximally
defined on $[0,T_{\text{terminal}})$.
\begin{lem}\label{lem:cont}
 The functions $y(t), \dot{y}(t)$, and $\ddot{y}(t)$ extend
continuously to $(0,T_{\text{terminal}})$. 
\end{lem}
\begin{proof} 
It suffices to check that the three functions are continuous at each
time $T_i$. Recall first the representation for $\tilde{u}$ from its
definition
\begin{equation}\label{eqn:repweak}\begin{aligned}
\tilde{u}(t,\cdot) &- \left[\sum_{k=1}^{N_i}\kappa_k^{(i)}
W_{\lambda_k^{(i)}(t)}(\cdot) + u_{i+1}(t,
\cdot)\right] \xrightarrow[t\nearrow T_{i+1}]{\dot{H}^1} 0~, \\
\tilde{u}_t(t,\cdot) &- u_{i+1,t}(T_{i+1},\cdot) 
\xrightarrow[t\nearrow T_{i+1}]{L^2} 0~.
\end{aligned}\end{equation}
Now observe that with a finite radius cutoff 
\[ \int_{|x| < R} W_\lambda^2\, \mathrm{d}x = 4\pi \int_0^R \frac{\lambda
r^2}{1 + \lambda^2 r^2/3} \, \mathrm{d}r \leq \frac{12\pi R}{\lambda} ~.\]
This implies that for each $\lambda_k^{(i)}$ we have 
\[ \sqrt{w} W_{\lambda_k^{(i)}(t)} \xrightarrow[t\nearrow
T_{i+1}]{L^2} 0 \]
and hence 
\[ \lim_{t\nearrow T_{i+1}} \langle w \tilde{u},\tilde{u}\rangle =
\langle w u_{i+1},u_{i+1} \rangle|_{t = T_{i+1}}\]
showing the continuity of $y(t)$. 

For the derivatives, we follow the computations in  \cite{CNLW1}. In particular, 
observe that 
\[
\dot{y}(t) = \langle \dot{w}\tilde{u} + 2w\dot{\tilde{u}}, \tilde{u}\rangle
\]
provided $t\in [T_{i}, T_{i+1})$. Now, the same argument as above
shows that using the uniformly bounded support of $\dot{w}$ and $w$
near $T_{i+1}$, 
\[ \lim_{t\nearrow T_{i+1}}\langle \dot{w}\tilde{u},\tilde{u}\rangle =
\langle \dot{w} u_{i+1},u_{i+1}\rangle|_{t = T_{i+1}}~.\]
Together with the $L^2$ convergence of $w\tilde{u} \to u_{i+1}(T_{i+1})$ and
$\dot{\tilde{u}} \to u_{i+1,t}(T_{i+1})$ as $t\nearrow T_{i+1}$
we get
\[
\lim_{t\nearrow T_{i+1}}\dot{y}(t) = \langle \dot{w}u_{i+1}(T_{i+1},
\cdot) + 2wu_{i+1, t}(T_{i+1}, \cdot),\,u_{i+1}(T_{i+1},
\cdot)\rangle\,.
\]
Since $\tilde{u}|_{[T_{i+1}, T_{i+2})} = u_{i+1}$, the continuity of $\dot{y}(t)$ across $t = T_{i+1}$ is evident. 
Next, consider $\ddot{y}(t)$, which according to  \cite{CNLW1} is given by the expression 
\begin{equation}\label{eq:key}
\ddot{y}(t) = \langle 2w,\,\dot{\tilde{u}}^2 - |\nabla \tilde{u}|^2 +
\tilde{u}^6\rangle + \langle \ddot{w}\tilde{u},\,\tilde{u}\rangle  +
\langle 4\dot{w}\tilde{u},\,\dot{\tilde{u}}\rangle - 2\langle
\tilde{u}\nabla w,\,\nabla \tilde{u}\rangle,\,t\in I_{i+1}\,.
\end{equation}
The continuity of the middle two terms at times $T_{i+1}$ is obtained
exactly as shown previously. For the last term, in addition to the
representation formulae \eqref{eqn:repweak} above we use also the fact that 
the derivative $\nabla w$ has compact spatial support uniformly (near
$t = T_{i+1}$) \emph{away} from the origin and so kills the contributions
from $\nabla W_{\lambda_k^{(i)}}(t)$. 

We examine the remaining term. The convergence in $L^2$ of
$\dot{\tilde{u}}$ to $\dot{u}_{i+1}$ as $t\nearrow T_{i+1}$ implies
\[
\lim_{t\nearrow T_{i+1}}\langle 2w(t,\cdot),\,\dot{\tilde{u}}^2(t,
\cdot)\rangle  = \langle 2w,\,u_{i+1,t}^2(T_{i+1}, \cdot)\rangle\,.
\]
The expressions $\langle 2w, |\nabla \tilde{u}|^2\rangle$ and $\langle
2w, \tilde{u}^6\rangle$ must be taken together as they \emph{do not
individually extend continuously across $T_{i+1}$}. Their difference,
however, does.  Indeed, it is straightforward to check that 
\[
\lim_{t\nearrow T_{i+1}}\langle 2w,\,|\nabla h_i|^2 - h_i^6\rangle =
0,\quad h_i(t, \cdot): = \sum_{k=1}^{N_i}\kappa_k^{(i)} W_{\lambda_k^{(i)}(t)}
\]
where we exploit of course the fact that $W$ is the ground state,
i.e.\ $\triangle W + W^5 = 0$, as well as the fact that the solitons
separate in scale, i.e.\ $\lambda_{k-1}^{(i)}\ll\lambda_{k}^{(i)}$. It follows that 
\[
\lim_{t\nearrow T_{i+1}}\langle 2w,\,- |\nabla \tilde{u}|^2 +
\tilde{u}^6\rangle = \langle 2w,\,- |\nabla u_{i+1}|^2 +
u_{i+1}^6\rangle |_{t = T_{i+1}}\,.
\]

The fact that $\ddot{y}(t)$ extends continuously across $T_{i+1}$
follows easily.
\end{proof}

We conclude this section with the following result, obtained as a
modification of the classical blow-up theorem of Levine \cite{Levine}.
\begin{lem}\label{lem:levine}
Let $\tilde{u}$ be a maximally extended canonical weak solution, and
suppose that at some positive time its energy $E(\tilde{u}) < 0$. Then
$T_{\text{terminal}} < +\infty$ for $\tilde{u}$. 
\end{lem}
\begin{proof}
Following \cite{CNLW1}, we observe that, due to the cut-off function
$w$ in \eqref{eq:defy}, we can write
\begin{equation}\label{eq:convex0}
\ddot{y}(t) = 2\big(4\|\dot{\tilde{u}}\|_{L^2}^2 +
4\|\nabla\tilde{u}\|_{L^2}^2 - 6E(\tilde{u})\big) + O(E_{\text{ext}})
\end{equation}
where  
\[
E_{\text{ext}}(t): = \int_{|x|>t+\tau}\big(|\dot{u}|^2 + |\nabla
u|^2\big)\,dx\lesssim E_{\text{ext}}(0)
\]
via a continuity argument and Huygens' principle, and the observation
that the bubbling off of solitons happen ``at the origin''. By picking
the initial cutoff $\tau > 0$ sufficiently large, we can force
$E_{\text{ext}}(0)$ as small as we want (as long as
$E_{\text{free}}(0)$ is finite). Suppose now (as given by the
hypothesis of our lemma) that for some $T_i$ that
$E(\tilde{u})|_{(T_i,T_{\text{terminal}})} < -2\varepsilon_* < 0$, where
we used the monotonicity of energy. Then a suitably large choice of $\tau$
would guarantee that
\begin{equation}\label{eq:convexT2} \ddot{y}(t) \geq 8
\|\dot{\tilde{u}}\|_{L^2}^2 + \eps_* \end{equation}
holds on $(T_i,T_{\text{terminal}})$.

Now assume, for contradiction, that $T_{\text{terminal}} = +\infty$. 
Note that \eqref{eq:convexT2} establishes a lower bound on
$\ddot{y}$ in $(T_i,\infty)$, which implies that after some large
finite time $\dot{y} > 0$ and $y > y_0$ is bounded below. Hence if one 
remarks just as in  \cite{CNLW1} that by Cauchy-Schwartz
\[ |\dot{y}(t)| = 2\langle w \dot{\tilde{u}},\tilde{u}\rangle +
O(E_{\text{ext}}) \leq 2 \sqrt{y} \|\dot{\tilde{u}}\|_{L^2}^{1/2} +
O(E_{\text{ext}}) \]
we have that at all sufficiently late times past $T_i$ we can upgrade
\eqref{eq:convexT2} to 
\begin{equation}\label{eq:convex2}
\ddot{y}(t)\geq \frac32 \frac{\dot{y}^2}{y} + \frac12 \eps_*~.
\end{equation}

From this inequality, however, we can apply the exact same argument as
in \cite{CNLW1}: \eqref{eq:convex2} implies that
$\frac{\mathrm{d}^2}{\mathrm{d}t^2} y^{-\frac12} < 0$ at all
$t > T_i$ sufficiently large; that $\dot{y}, y > 0$ implies that
$\frac{\mathrm{d}}{\mathrm{d}t}y^{-\frac12} < 0$ at all $t > T_i$
sufficiently large. Together the concavity implies $y(t)$ must blow up 
in finite time, ruling out the possibility $T_{\text{terminal}} =
+\infty$ and proving our lemma. 
\end{proof}

\begin{thm}\label{thm:km2}
The maximally extended canonical weak solution for the Kenig-Merle
type \cite{KM2} blow-up initial data with $E(u) < E(W)$ and
$\|\nabla_x u\|_{L^2} > \|\nabla W\|_{L^2}$ terminates in finite
$T_{\text{terminal}}$ with the type I condition
\eqref{eq:weaktypeiblowup} satisfied. 
\end{thm}
\begin{proof}
If $T_1$ ends in a type I blow-up, we are done. If not, by the profile
decomposition and \eqref{eq:energyjump} we have that
$E(\tilde{u})(T_1) = E(u)(0) - N_1\cdot E(W) < 0$. The theorem then
follows from Lemmas \ref{lem:fti} and \ref{lem:levine}. 
\end{proof}

\section{Formation of type I singularities: above threshold solutions}

\changed{Theorem} \ref{thm:km2} \changed{above settles the problem for initial data with
energy below that of the ground state, in view of the dichotomy proven
in} \cite{KM2}. \changed{We now turn our attention to
whether there exist \emph{generic sets} (not necessarily in the energy
topology) of solutions which satisfy $T_{\text{terminal}} < +\infty$
with initial energy above that of the ground state. We note that by
appropriately time-translating the type II blow-up solutions 
constructed in} \cite{KST}, \changed{we obtain one that satisfies $T_1 > 0$ and
$T_2 = +\infty$. In order to rule out these type of behaviour, we move
to a stronger topology}\footnote{This is analogous to \cite{KS}, where 
conditional stability of the ground state $W$ is shown for a stronger
topology than energy. The same question is open in energy topology. We
refer the readers to \cite{CNLW3} for a summary.}: \changed{here we review the 
results of} \cite{CNLW3}. 

\changed{First we recall that linearising} \eqref{eq:Main} 
\changed{around the solution
$W$ leads us to consider the linearised operator $-\triangle - 5 W^4$. On 
radial functions, this linearised operator has 
a unique negative eigenvalue $-k_d^2$ with eigenfunction $g_d$ satisfying
$g_d > 0$; this
contributes to the linear instability of the ground state $W$. 
In} \cite{KS}, \changed{it was shown, for initial data supported in a fixed ball 
with the topology $H^3_{\text{rad}}(\mathbb{R}^3) \times
H^2_{\text{rad}}(\mathbb{R}^3)$, that there exists a Lipschitz
manifold $\Sigma$ in a small neighbourhood of the ground state
$W$, which contains the soliton curve $\mathcal{S}$ (i. e. rescalings of $W$),
such that initial data given on $\Sigma$ exists globally and scatters
to $\mathcal{S}$. Moreover, this Lipschitz manifold $\Sigma$ is
transverse to $(g_d,0)$: indeed, $\Sigma$ is written as a Lipschitz graph
over the subspace orthogonal to $g_d$ of the tangent space at $W$.
Therefore an $\eps$-neighbourhood of $W$ can be divided
into the portion `above' $\Sigma$ (i. e. those that can be written
as $\sigma + \delta(g_d,0)$ for $\sigma\in \Sigma$ and $0 < \delta <
\eps$) 
and those `below' $\Sigma$ (with a minus sign instead). In}
\cite{CNLW3}, \changed{it was shown that this division provides a dichotomy:
those data sitting above $\Sigma$ blows up in finite time, while those
data sitting below $\Sigma$ has global existence in forward time and
scatters to zero in energy space.}

Our main theorem concerns the blow-up solutions sitting above
$\Sigma$:
\begin{thm}\label{thm:main} 
\changed{Let $u(t, \cdot)$ be one of the blow up
solutions with initial data of the form $\sigma + \delta(g_d,0)$ with
$\delta > 0$ as described above. Then the
canonical weak extension of this solution 
will satisfy $T_{\text{terminal}}< +\infty$. Furthermore,
the canonical weak solutions satisfy} \eqref{eq:weaktypeiblowup}.
\end{thm}

\begin{proof}
By Lemma \ref{lem:fti} it suffices to rule out the case 
$T_{\text{terminal}} = +\infty$. Now, if $T_2
< T_{\text{terminal}}$ or $T_1 < T_{\text{terminal}}$ with $N_i > 1$,
by the energy jump condition \eqref{eq:energyjump} we have that the
conditions of Lemma \ref{lem:levine} is satisfied, since our initial
energy is close to that of a single soliton, and thus
$T_{\text{terminal}} < +\infty$. 

It remains to rule out the case where $T_1 < T_2 =
T_{\text{terminal}}= +\infty$, where \emph{exactly one} soliton has bubbled off
at $T_1$. \changed{For this, we will appeal to the one pass theorem of}
\cite{CNLW1}, \changed{which states roughly that, for initial data
close to the soliton curve $\mathcal{S}$, once the solution leaves a
small neighbourhood of $\mathcal{S}$ it can never return. More
precisely,} we can write
\[ u(t,\cdot) \approx \kappa W_{\lambda(t)} + u_1(t,\cdot)~,\,\kappa\in \{\pm 1\},\,
t\in [0,T_1)~.\]
For $t \in [T_1,T_2)$, the energy satisfies
\[ E(\tilde{u})(t) = \int_{\R^3}\big(\frac{1}{2}(\tilde{u}_t)^2 +
|\nabla\tilde{u}|^2) - \frac{1}{6}\tilde{u}^6\big)\,dx<\eps~,\]
and by using Sobolev's inequality we get that for some constant $C_* >
0$
\[ \frac12 \|\nabla \tilde{u}\|_{L^2}^2 - C_* \|\nabla
\tilde{u}\|_{L^2}^6 < \eps \]
which implies that if the constant $\eps$ (which we recall
measures \changed{the distance from the soliton curve of our initial
data}) is chosen sufficiently small, by continuity we
must have that throughout $t\in [T_1,T_2)$, either
\begin{subequations}
\begin{equation}\label{badcase} \|\nabla\tilde{u}\|_{L^2} \lesssim \sqrt{\eps}
\end{equation}
or
\begin{equation}\label{goodcase} \|\nabla\tilde{u}\|_{L^2} \gtrsim 1~.\end{equation}
\end{subequations}

\changed{We rule out the case} \eqref{badcase}:
\changed{it} would necessarily require a bound 
\[ \|\tilde{u}_t\|_{L^2} \lesssim \sqrt{\eps} \]
which implies that 
\[ \|\nabla_{t,x}u_1(T_1,\cdot)\|_{L^2_x} \lesssim \sqrt{\eps}~. \]
\changed{This requires that there exists $\tilde{t}$ less than but 
arbitrarily close to $T_1$ such that the inequality}
\begin{subequations}
\begin{equation}\label{closeto} \mathrm{dist}_{\dot{H}^1\times L^2}(u[\tilde{t}], \mathcal{S}\cup
-\mathcal{S}) \lesssim\sqrt{\eps}\end{equation}
\changed{holds. But from Proposition 1.2 and the proof of Theorem 1.1
in} \cite{CNLW3}
\changed{we see that, assuming $\eps > 0$ is sufficiently small, for
some $t \in [0,T_1)$ we must have}
\begin{equation}\label{awayfrom} \mathrm{dist}_{\dot{H}^1\times L^2}(u[t],\mathcal{S}\cup -
\mathcal{S}) \gg \sqrt{\epsilon} \end{equation}\end{subequations}
\changed{due to the exponential growth of the unstable mode. The two
equations} \eqref{closeto} \changed{and} \eqref{awayfrom} \changed{are
contradictory in view of Theorem 4.1 in} \cite{CNLW1}. 

\changed{It follows that the alternative} \eqref{goodcase} \changed{must hold. 
Looking} at \eqref{eq:convex0} again we see that \changed{if $\tau$ is chosen 
sufficiently large, and if $\eps$ is
sufficiently small, the expression} \eqref{eq:convexT2} \changed{would also
apply in $[T_1,T_2)$ for a suitable $\eps_*$.} We can then conclude exactly as in the proof of
Lemma \ref{lem:levine}.  
\end{proof}

\centerline{\scshape Joachim Krieger }
\medskip
{\footnotesize
 \centerline{B\^{a}timent des Math\'ematiques, EPFL}
\centerline{Station 8, 
CH-1015 Lausanne, 
  Switzerland}
  \centerline{\email{joachim.krieger@epfl.ch}}
} 

\medskip

\centerline{\scshape Willie Wong}
\medskip
{\footnotesize
 \centerline{B\^{a}timent des Math\'ematiques, EPFL}
\centerline{Station 8, 
CH-1015 Lausanne, 
  Switzerland}
  \centerline{\email{willie.wong@epfl.ch}}
} 

\bigskip

\end{document}